\documentclass[11pt]{article}
\usepackage{amsmath,amssymb,amsthm,amscd,dsfont}
\usepackage{amsfonts,latexsym,rawfonts,amsmath,amssymb,amsthm}
\usepackage{lscape}
\usepackage{amscd, float,times,rotating}
\usepackage{pb-diagram}
\usepackage{hyperref}
\numberwithin{equation}{section}

\newtheorem{prop}{Proposition}[section]
\newtheorem{theorem}[prop]{Theorem}
\newtheorem{lemma}[prop]{Lemma}

\newtheorem{example}[prop]{Example}

\begin{document}
\title{Asymptotic behavior of the first Dirichlet eigenvalue of AHE manifolds}
\author{Xiaoshang Jin}
\date{}
\maketitle
\begin{abstract}
In this article, we investigate the rate at which the first Dirichlet eigenvalue of geodesic balls decreases as the radius approaches infinity. We prove that if the conformal infinity of an asymptotically hyperbolic Einstein manifold is of nonnegative Yamabe type, then the two-term asymptotic of the eigenvalues is the same as that in hyperbolic space.
\end{abstract}
\section{Introduction}
Suppose that $\mathbb{H}^{n+1}$ is an $n+1-$ dimensional hyperbolic space, then it is well-known that the spectrum of the Laplacian
$\sigma(-\Delta_\mathbb{H})=[\frac{n^2}{4},+\infty).$ Later the result is extended by R. Mazzeo in \cite{mazzeo1988hodge} and \cite{mazzeo1991unique}. He showed that if $(X,g)$ is an $n+1-$ dimensional asymptotically hyperbolic manifold, then the spectrum of the Laplacian
$$\sigma(-\Delta)=\sigma_p(-\Delta)\cup[\frac{n^2}{4},+\infty).$$
Here $\Delta=\Delta_{g}=\frac{1}{\sqrt{G}}\frac{\partial}{\partial x^i}(g^{ij}\sqrt{G}\frac{\partial}{\partial x^j})$ stands
for the Laplace-Beltrami operator and $\sigma_p(-\Delta)\subseteq (0,\frac{n^2}{4})$ is a finite set of point spectrums ($L^2$ eigenvalues). Lee \cite{lee1994spectrum} discovered a connection between its spectrum and conformal infinity when $g$ is also Einstein. In other words, $\sigma_p(-\Delta)$ is empty when $(X,g)$ is an asymptotically hyperbolic Einstein (AHE) manifold with conformal infinity of nonnegative Yamabe type. One can also see \cite{wang2002new} for another proof.
\par We rewrite Lee's result: $Y(\partial X,[\hat{g}])\geq 0\Rightarrow \lambda_1(X)=\frac{n^2}{4}.$  It is clear from the property of the first Dirichlet eigenvalue on the noncompact manifold that
$$\lim\limits_{R\rightarrow+\infty} \lambda_1(B(p,R))=\frac{n^2}{4}$$
holds for any $p\in X.$ Here $B(p,R)$ is the geodesic ball centered at $p$ of radius $R.$ In this paper, we will present the rate of how $\lambda_1(B(p,R))$ tends  to $\frac{n^2}{4}.$
\par Before stating our primary theorem, let's first present some basic conceptions. Suppose that $\overline{X}$ is a compact manifold with smooth boundary $\partial X$ and $g$ is a complete metric in its interior $X.$ We say that $(X,g)$ is conformally compact if there exists a defining function $\rho$ such that $\bar{g}=\rho^2 g$ extends continuously to $\overline{X}.$ Here
$$\rho>0\ in\ X,\ \ \ \rho=0\ on\ \partial X,\ \ \ d\rho\neq 0\ on\ \partial X.$$
$(X,g)$ is called $C^{m,\alpha}$ (smoothly) conformally compact if $\bar{g}=\rho^2 g$ is $C^{m,\alpha}$ (smooth) on $\overline{X}.$ For any defining function $\rho,$ we call $\hat{g}=\rho^2g|_{T\partial X}$ the boundary metric. Hence the conformal class $(\partial X,[\hat{g}])$ is uniquely determined by $g$ and we call it the conformal infinity of $g.$
\par Let $\bar{g}=\rho^2g$ be a $C^2$ conformal compactification of $(X,g),$ then a simple calculation, such as that in \cite{mazzeo1988hodge} indicates that the sectional curvature of $g$ tends to $-|d\rho|^2_{\bar{g}}|_{\partial X}$ as $\rho\rightarrow 0.$ Therefore no matter what the topology and geometry of $g$ look like in $X,$ the boundary behavior would always remind us of hyperbolic space. We call $(X,g)$ an asymptotically hyperbolic (AH for short) manifold if it is  conformally compact and $|d\rho|^2_{\bar{g}}=1$ on the boundary $\partial X.$
\par Let $(X,g)$ be a $C^2$ conformally compact manifold of dimension $n+1,$ if $g$ is also Einstein, i.e. $Ric[g]=-ng,$ then $|d\rho|^2_{\rho^2g}=1$ on $\partial X$ for any smooth defining function $\rho.$ In this case, we say that $(X,g)$ is an asymptotically hyperbolic Einstein (AHE for short) manifold.
\par Here is the main result of this paper:
\begin{theorem}
  Let $(X,g)$ be an $n+1-$ dimensional AHE manifold with conformal infinity $(\partial X,[\hat{g}]).$ If the Yamabe constant $Y(\partial X,[\hat{g}])\geq 0,$ then for any $p\in X,$
\begin{equation}\label{1.1}
  \lambda_1(B(p,R))=\frac{n^2}{4}+\frac{\pi^2}{R^2}+O(R^{-3}),\ \ R\rightarrow+\infty.
\end{equation}
  Here $\lambda_1$ denotes the first Dirichlet eigenvalue.
\end{theorem}
Theorem 1.1 makes clear that the rate at which the first Dirichlet eigenvalue of geodesic balls tends to $\frac{n^2}{4}$ is the same as that in hyperbolic space, at least in the second term of the expansion.
\par On the other hand, we believe that the rate at which the first Dirichlet eigenvalue decreases is related to the geometry structure of manifolds. It is connected to the number of ends. Let's recall the work of Witten-Yau \cite{witten1999connectedness}. They showed that the boundary $\partial X$ of an AHE manifold $(X,g)$ is connected if $Y(\partial X,[\hat{g}])>0.$ Later the work was extended by Cai and Galloway in \cite{cai2000boundaries} where they relaxed the assumption that $\partial X$ has nonnegative Yamabe constant.
In \cite{wang2001conformally}, Wang proved that if the first eigenvalue of an AHE manifold $\lambda_1(X)\geq n-1,$ then it either has only one end or it must be a warped product $(\mathbb{R}\times N, dt^2+\cosh^2 t g_N).$ It would provide a new proof for Cai-Galloway's result if combined with Lee's work in \cite{lee1994spectrum}. Let's summarize their work: for an AHE manifold $(X,g),$
$$Y(\partial X, [\hat{g}])\geq 0\Longrightarrow \lambda_1(X)=\frac{n^2}{4}\Longrightarrow \partial X \ is\ connected\ (X\ has\ one\ end).$$
Later, Li and Wang expanded the results in \cite{li2001complete} and \cite{li2002complete} where they didn't require $X$ to be conformally compact. In this case, $X$ either has one end or is a warped product. Now we could rule out the case of warped product by a direct calculation.
\par In fact, as an application of theorem 0.5 and 0.6 in \cite{li2002complete}, we could obtain the following property:
\begin{prop}
Let $(X,g)$ be a complete $n+1-$dimensional manifold with $n\geq 2$ and $Ric[g]\geq -ng.$ If
\begin{equation}\label{1.2}
\lambda_1(B(p,R))=\frac{n^2}{4}+\frac{\pi^2}{R^2}+O(R^{-3}), \ \  R\rightarrow+\infty
\end{equation}
for some $p\in X,$ then $X$ has only one end with infinite volume.
\end{prop}
This paper is organized as follows.  In section 2, we first provide some background information on the Dirichlet eigenvalue. Then in sections 3 and 4, we prove theorem 1.1. In order to get the upper bound of the first Dirichlet eigenvalue of geodesic balls, we use the eigenvalue comparison theory and the eigenvalue formula in hyperbolic space. To estimate the lower bound, we somewhat enhance Lee's work. To be more precise, we create a new test function $u^{-\frac{n}{2}}\cdot\sin(a\ln\varepsilon u)$ on a bounded domain. Here $u$ is the eigenfunction solution to $\Delta u=(n+1)u$ and was first used by Lee in \cite{lee1994spectrum}. In the end, we prove proposition 1.2 in section 5.
\section{The first Dirichlet eigenvalue of manifolds}
Let's introduce some materials about Dirichlet eigenvalue in this section. Suppose that $(M,g)$ is a complete manifold and $\Omega\subseteq M$ is a bounded domain of $M$ with piecewise smooth boundary. The Dirichlet eigenfunctions are defined by solving the following problem for $u\neq 0$ and eigenvalue $\lambda.$
 \begin{equation}
 \left\{
    \begin{array}{l}
    \Delta u=-\lambda u \ \ \ \ in \ \Omega,
    \\ u=0\ \ \ \ \ \ \ \ \ \ \ \ \ on\ \partial \Omega
    \end{array}
 \right.
 \end{equation}
 where $\Delta=\frac{1}{\sqrt{G}}\frac{\partial}{\partial x^i}(g^{ij}\sqrt{G}\frac{\partial}{\partial x^j}).$ The smallest eigenvalue is denoted
 by $\lambda_1=\lambda_1(\Omega)>0.$ Recall the Sobolev space
 $H^1(\Omega)=W^{1,2}(\Omega)$ and $H^1_0(\Omega)\subseteq H^1(\Omega)$ is defined to be the closure of the infinitely differentiable functions compactly supported in $\Omega.$ Then by the max-min principle,
 \begin{equation}
   \lambda_1(\Omega)=\inf\limits_{f\in H^1_0(\Omega)\setminus\{0\}}\frac{\int_\Omega |\nabla f|^2dV_g}{\int_\Omega  f^2dV_g}
 \end{equation}
It's easy to see that the eigenvalue has domain monotonicity: if $\Omega_1\subseteq\Omega_2\Subset M,$ then $ \lambda_1(\Omega_1)\geq  \lambda_1(\Omega_2).$
 \par Now we suppose that $(M,g)$ is a noncompact manifold, and denote the greatest lower bound for the $L^2$-spectrum of the Laplacian by
 \begin{equation}
 \lambda_1(M):=\inf spec(-\Delta)=\inf\limits_{f\in H^1_0(M)\setminus\{0\}}\frac{\int_M |\nabla f|^2dV_g}{\int_M  f^2dV_g}.
  \end{equation}
 Notice that $\lambda_1(M)$ does not necessarily be an $L^2$ eigenvalue of $-\Delta,$
but is motivated by the characterization by
 \begin{equation}
 \lambda_1(M)=\lim\limits_{k\rightarrow\infty}\lambda_1(\Omega_k)
 \end{equation}
 for any smoothly compact exhaustion $\{\Omega_k\}$ of $M.$
 \par For example, for the hyperbolic space $M=\mathbb{H}^{n+1},$ we know that $spec(-\Delta)=[\frac{n^2}{4},+\infty),$ see \cite{mckean1970upper}. Then for any $p\in \mathbb{H}^{n+1},$
 \begin{equation}
 \lim\limits_{R\rightarrow+\infty}\lambda_1(B(p,R))=\frac{n^2}{4}.
  \end{equation}
 It is an interesting problem what the formula of $\lambda_1(B(p,R)$ looks like. Or how $\lambda_1(B(p,R))$ tends to $\frac{n^2}{4}?$
It is shown in \cite{savo2009lowest} and \cite{artamoshin2016lower} that
 \begin{equation}
  \lambda_1(B(p,R))=\frac{n^2}{4}+\frac{\pi^2}{R^2}+O(R^{-3}),\ \ R\rightarrow+\infty.
\end{equation}
In this paper, we prove that (2.6) still holds for AHE manifolds with conformal infinity of nonnegative Yamabe type.

\section{The upper bound of eigenvalues}
Let's recall the classic eigenvalue comparison theorem of Cheng \cite{cheng1975eigenvalue}. If $(X,g)$ is an $n+1$ dimensional complete manifold satisfying that $Ric[g]\geq -n g,$ then for any $p\in X$ and $R>0,$ $\lambda_1(B(p,R))\leq \lambda_1(B^\mathbb{H}(R)).$ Here
$B^\mathbb{H}(R)$ is a geodesic ball of radius $R$ in hyperbolic space. He also showed that $\lambda_1(B^\mathbb{H}(R))\leq\frac{n^2}{4}+\frac{C}{R^2}$ for some positive constant $C.$ Later the upper bound estimate was extended by Gage, see theorem 5.2 in \cite{gage1980upper}. In the following, we provide a weak version of the estimate for the upper bound and the proof is also simpler.
\begin{theorem}
  Let $\mathbb{H}^{n+1}$ be the hyperbolic space of $n+1$ dimension, then for any $p\in \mathbb{H}^{n+1},$
\begin{equation}
\lambda_1(B(p,R))\leq \frac{n^2}{4}+\frac{\pi^2}{R^2}+O(R^{-3})\ \  \ R\rightarrow+\infty.
\end{equation}
\end{theorem}
\begin{proof}
  Consider the rotationally symmetric model:
  $$(\mathbb{R}^{n+1},g_\mathbb{H}=dt^2+\sinh^2t g_\mathbb{S})$$
  and let $p$ be the center point. For any $R>0,$ we define the function
\begin{equation}
f=e^{-\frac{n}{2}t}\cdot \sin (\frac{\pi}{R} t)\in H^1_0((B(p,R))
\end{equation}
  Then
  \begin{equation}
    \begin{aligned}
       \lambda_1(B(p,R))&\leq \frac{\int_{B(p,R)}|\nabla f|^2 dV[g^\mathbb{H}]}{\int_{B(p,R)}f^2 dV[g^\mathbb{H}]}
       \\ &= \frac{\int_0^R e^{-nt}(-\frac{n}{2}\sin (\frac{\pi}{R}t)+\frac{\pi}{R}\cos(\frac{\pi}{R} t))^2 \cdot \omega_n\sinh^n tdt}
              {\int_0^R e^{-nt} \sin^2 (\frac{\pi}{R} t) \cdot \omega_n\sinh^n tdt}
       \\&=\frac{\int_0^R (1-e^{-2t})^n\cdot (-\frac{n}{2}\sin (\frac{\pi}{R}t)+\frac{\pi}{R}\cos(\frac{\pi}{R} t))^2 dt}
              {\int_0^R (1-e^{-2t})^n\cdot \sin^2 (\frac{\pi}{R} t)dt}
       \\&=\frac{\int_0^\pi (1-e^{-\frac{2R\theta}{\pi}})^n\cdot (-\frac{n}{2}\sin \theta+\frac{\pi}{R}\cos \theta)^2 d\theta}
              {\int_0^\pi (1-e^{-\frac{2R\theta}{\pi}})^n\cdot \sin^2 \theta d\theta}
       \\&=\frac{F(R)}{G(R)}
    \end{aligned}
  \end{equation}
 where
  $$F(R)\leq\int_0^\pi (-\frac{n}{2}\sin \theta+\frac{\pi}{R}\cos \theta)^2 d\theta=\frac{\pi}{2}\cdot(\frac{n^2}{4}+\frac{\pi^2}{R^2})$$
  For the term $G(R),$ a direct calculation indicates that
 \begin{equation}
       \int_0^\pi e^{-r\theta}\cdot \sin^2 \theta d\theta=\frac{2}{r(r^2+4)}(1-e^{-\pi r})=\frac{2}{r^3}+O(r^{-4}),\ r\rightarrow+\infty
 \end{equation}
 Hence we could get that
  \begin{equation}
    \begin{aligned}
      G(R)&=\int_0^\pi [1+\sum\limits_{k=1}^n C_n^k(-e^{-\frac{2R\theta}{\pi}})^k]\sin^2\theta d\theta
        \\&=\frac{\pi}{2}-\frac{\pi^3}{4}[\sum\limits_{k=1}^n C_n^k\frac{(-1)^{k+1}}{k^3}]\frac{1}{R^3}+O(R^{-4})
    \end{aligned}
  \end{equation}
 In the end, we deduce that
   \begin{equation}
       \lambda_1(B(p,R))\leq \frac{F(R)}{G(R)}\leq \frac{n^2}{4}+\frac{\pi^2}{R^2}+\frac{n^2\pi^2}{8}[\sum\limits_{k=1}^n C_n^k\frac{(-1)^{k+1}}{k^3}]\frac{1}{R^3}+O(R^{-4}).
   \end{equation}

\end{proof}
\section{The lower bound of eigenvalues}
Suppose that $(X,g)$ satisfies the conditions in theorem 1.1. Lee proved that $\lambda_1(X)=\frac{n^2}{4}$ in \cite{lee1994spectrum}. The key step is to construct a proper test function $\varphi=u^{-\frac{n}{2}}.$ Here $u$ is an important positive eigenfunction with prescribed growth at infinity. In order to make our proof more clear, we would give a short quick introduction to Lee's proof.
\subsection{A quick review of Lee's work}
\begin{lemma}\cite{lee1994spectrum}
Let $(X,g)$ be an $n+1-$ dimensional AHE manifold with boundary metric $\hat{g}$ and let $x$ be the associated geodesic defining function. Then there is a unique positive eigenfunction $u$ on $X$ such that
$$\Delta u=(n+1)u.$$
and $u$ has the following form of expansion at the boundary
$$u=\frac{1}{x}+\frac{\hat{R}}{4n(n-1)}x+O(x^2).$$
\end{lemma}
Let $\varphi=u^{-\frac{n}{2}},$ then
$$-\frac{\Delta\varphi}{\varphi}=\frac{n^2}{4}+\frac{n(n+2)}{4}(1-\frac{|du|^2_{g}}{u^2}).$$
One can estimate near the boundary:
$$u^2-|du|^2_{g}=\frac{\hat{R}}{n(n-1)}+o(1).$$
On the other hand, the Bochner formula implies that
$$-\Delta(u^2-|du|^2_{g})=2|\frac{\Delta u}{n+1}g+\nabla^2 u|^2\geq 0.$$
When $Y(\partial X,[\hat{g}])\geq 0$ we can choose a representative $\hat{g}$ whose scalar curvature $\hat{R}\geq 0,$ then the maximum principle implies that $u^2-|du|^2_{g}\geq 0$ in $X.$ So
$-\frac{\Delta\varphi}{\varphi}\geq \frac{n^2}{4}$ in $X.$ According to the eigenvalue comparison theorem of Cheng-Yau \cite{cheng1975eigenvalue}, $\lambda_1(X)\geq \frac{n^2}{4}.$
\\
\par Now we turn to research the first Dirichlet eigenvalue of geodesic balls. For sufficiently small $\varepsilon>0,$ let $X_\varepsilon=\partial X\times(0,\varepsilon).$ We study the first Dirichlet eigenvalue of $X\setminus X_\varepsilon.$ If $\hat{R}>0,$ we get that
$$1-\frac{|du|^2_{g}}{u^2}\geq\frac{\hat{R}}{n(n-1)}\frac{1}{u^2}+o(\frac{1}{u^2})=\frac{\hat{R}}{n(n-1)} x^2+o(x^2).$$
Then
$$-\frac{\Delta\varphi}{\varphi}\geq \frac{n^2}{4}+c\varepsilon^2,\ \ \ on\ X\setminus X_\varepsilon$$
for some positive constant $c$ and hence $\lambda_1(X\setminus X_\varepsilon)\geq\frac{n^2}{4}+c\varepsilon^2.$
As a consequence,
\begin{equation}
\lambda_1(B(p,R))\geq \frac{n^2}{4}+\frac{C}{e^{2R}}
\end{equation}
for some $C>0$ provided $R$ is large enough. If $\hat{R}=0,$ we know that $1-\frac{|du|^2_{g}}{u^2}$ is still positive in $X,$ see \cite{guillarmou2010spectral}. Then a similar estimate of (4.1) could be obtained. The lower bound $\frac{C}{e^{2R}}$ is too "small" compared to $\frac{\pi^2}{R^2}.$
 We need to find a better test function to get a sharper lower bound of $\lambda_1(B(p,R)).$
\subsection{A new test function}
Let $u$ be the eigenfunction that is defined in lemma 4.1 and $\varphi=u^{-\frac{n}{2}}.$ In the following, for sufficiently small $\varepsilon>0,$ we consider a new test function
\begin{equation}
\psi=\varphi\cdot h=u^{-\frac{n}{2}}\cdot\sin(a\ln\varepsilon u)
\end{equation}
on the bounded domain
\begin{equation}
F_\varepsilon=\{p\in X:u(p)<\frac{1}{\varepsilon}\}
\end{equation}
where $a=a(n,\varepsilon)<0$ is a constant to be determined.
A simple calculation indicates that
\begin{equation}
h'=\frac{a\cos(a\ln\varepsilon u)}{u},\ \ h''=\frac{-a^2\sin(a\ln\varepsilon u)-a\cos(a\ln\varepsilon u)}{u^2}=-a^2\frac{h}{u^2}-\frac{h'}{u}.
\end{equation}
Hence
\begin{equation}
\Delta h=h'\Delta u+h''|du|^2=(n+1)uh'-(a^2 h+uh')\frac{|du|^2}{u^2}
\end{equation}
and
\begin{equation}
2g(d\ln \varphi,d\ln h)=2g(-\frac{n}{2}\frac{du}{u},\frac{h'}{h}du)=-n\frac{uh'}{h}\frac{|du|^2}{u^2}.
\end{equation}
As a consequence,
\begin{equation}
  \begin{aligned}
  -\frac{\Delta\psi}{\psi}&=-\frac{\Delta\varphi}{\varphi}-(\frac{\Delta h}{h}+2g(d\ln \varphi,d\ln h)
  \\ &=\frac{n^2}{4}+\frac{n(n+2)}{4}(1-\frac{|du|^2}{u^2})-
  [(n+1)\frac{uh'}{h}-(a^2+\frac{uh'}{h})\frac{|du|^2}{u^2}-n\frac{uh'}{h}\frac{|du|^2}{u^2}]
  \\ &=\frac{n^2}{4}+a^2+(1-\frac{|du|^2}{u^2})[\frac{n(n+2)}{4}-(n+1)\frac{uh'}{h}-a^2]
  \\ &=\frac{n^2}{4}+a^2+(1-\frac{|du|^2}{u^2})[\frac{n(n+2)}{4}-(n+1)a\cdot\cot(a\ln\varepsilon u)-a^2]
  \end{aligned}
\end{equation}
We could assume that $u\geq 1 $ on $X,$ or else we use $ku$ instead where $k$ is a constant large enough.  Now set
\begin{equation}
a=\frac{\pi}{\ln \varepsilon}+\frac{c_n}{\ln^2\varepsilon}
\end{equation}
for some constant $c_n>0,$ then $a<0.$ Hence on $F_\varepsilon,$ we have that
$$a\ln\varepsilon u\in(0,\pi+\frac{c_n}{\ln \varepsilon}]\subseteq(0,\pi).$$
As a result, $h$ is smooth and positive on $F_\varepsilon$ and so is $\psi.$ Furthermore,
\begin{equation}
  \begin{aligned}
  -a\cdot\cot(a\ln\varepsilon u)&\geq-a\cdot\cot(\pi+\frac{c_n}{\ln\varepsilon})=-a\frac{\cos(\pi+\frac{c_n}{\ln\varepsilon})}{\sin(\pi+\frac{c_n}{\ln\varepsilon})}
  \\&>\frac{a}{\sin(\pi+\frac{c_n}{\ln\varepsilon})}=\frac{\frac{\pi}{\ln \varepsilon}+\frac{c_n}{\ln^2\varepsilon}}{\sin(-\frac{c_n}{\ln\varepsilon})}\rightarrow -\frac{\pi}{c_n}.
  \end{aligned}
\end{equation}
Therefore
\begin{equation}\label{4.10}
\liminf\limits_{\varepsilon\rightarrow 0}[\frac{n(n+2)}{4}-(n+1)a\cdot\cot(a\ln\varepsilon u)-a^2]\geq \frac{n(n+2)}{4}-(n+1)\frac{\pi}{c_n}.
\end{equation}
If we choose $c_n\geq\frac{4\pi(n+1)}{n(n+2)},$ then the formula (\ref{4.10})is nonnegative and finally we can get that
\begin{equation}
-\frac{\Delta\psi}{\psi}\geq\frac{n^2}{4}+a^2
\end{equation}
on $F_\varepsilon$ provided $\varepsilon$ is sufficiently small. Then
\begin{equation}
\lambda_1(F_\varepsilon)\geq\frac{n^2}{4}+a^2=\frac{n^2}{4}+\frac{\pi^2}{\ln^2\varepsilon}+O(\frac{1}{\ln^3\varepsilon}).
\end{equation}
For any $p\in X$ and large $R>0,$ let's consider the first Dirichlet eigenvalue of $B(p,R).$  Since
\begin{equation}
|u-\frac{1}{x}|\leq C_1,\ \ \ |-\ln x(\cdot)-d_{g}(p,\cdot)|\leq C_2
\end{equation}
where $C_1$ and $C_2$ are positive constants, we have that
\begin{equation}
e^{d_{g}(p,\cdot)}\geq \frac{e^{-C_2}}{x}\geq e^{-C_2}(u-C_1)
\end{equation}
Then
$$B(p,R)\subseteq\{q\in X:u(q)< e^{R+C_2}+C_1\}\subseteq F_{e^{-R-C_3}}
$$
for some constant $C_3>0$ when $R$ is large enough.
Hence
\begin{equation}\label{4.15}
  \begin{aligned}
    \lambda_1(B(p,R)&\geq \lambda_1(F_{e^{-R-C_3}})
    \\ & \geq\frac{n^2}{4}+\frac{\pi^2}{(R+C_3)^2}+O(\frac{1}{(R+C_3)^3})
    \\ &=\frac{n^2}{4}+\frac{\pi^2}{R^2}+O(\frac{1}{R^3}).
  \end{aligned}
\end{equation}

Proof of theorem 1.1: theorem 3.1 and the eigenvalue comparison theorem in \cite{cheng1975eigenvalue} provide the upper bound of the first Dirichlet eigenvalue of balls while (\ref{4.15}) provides the lower bound. Then we finish the proof for theorem 1.1.
\section{The geometric property of the asymptotical behavior}
To prove proposition 1.2, we introduce an important result of Li and Wang:
\begin{theorem}\label{thm5.1}\cite{li2002complete}
  Let $(M,g)$ be an $n+1$ dimensional complete manifold with $n\geq 2.$ Suppose that $Ric[g]\geq -n$ and $\lambda_1(M)=\frac{n^2}{4}.$ Then
  \par (1) $M$ has only one end with infinite volume; or
  \par (2) $(M,g)=(\mathbb{R}\times N,dt^2+e^{2t}g_N)$ where $(N,g_N)$ is an $n-$dimensional compact manifold satisfying that $Ric[g_N]\geq 0;$ or
  \par (3) $n=2$ and $(M,g)=(\mathbb{R}\times N,dt^2+\cosh^2 tg_N)$ where $N$ is a compact surface satisfying that the sectional curvature $K_N\geq -1.$
\end{theorem}
In the following, we will show that the rate of eigenvalues in case (2) and (3) of theorem 5.1  does not match the formula (\ref{1.2}).
\begin{example}
Let $(M,g)=(\mathbb{R}\times N,dt^2+e^{2t}g_N)$ be an $n+1-$dimensional manifold $(n\geq 2)$ where $(N,g_N)$ is an $n-$dimensional compact manifold satisfying that $Ric[g_N]\geq 0.$ Then
$Ric[g]\geq -n$ and $\lambda_1(M)=\frac{n^2}{4}.$
\end{example}
Now we are going to study the formula of $\lambda_1(B(p,R)).$ We assume that $d_N=diam (N,g_N)>0$ and for any $p\in M$ and large $R>0,$ let
$d_p=dist_g(p,N).$ Then
\begin{equation}\label{5.1}
  B(p,R-d_p)\subseteq B(N,R)\subseteq B(p,R+d_N+d_p).
\end{equation}
Here  $B(N,R)=(-R,R)\times N.$ Let $f=e^{-\frac{n}{2}t}\cos(\frac{\pi}{2R}t),$ then
$$f>0\ \ in\ B(N,R),\ \ \ \ f=0 \ \ on\ \partial B(N,R).$$
On the other hand,
\begin{equation}
  \begin{aligned}
    \Delta f&=f'(t)\Delta t+f''(t)|\nabla t|^2=nf'(t)+f''(t)
    \\ &=n e^{-\frac{n}{2}t}[-\frac{n}{2}\cos(\frac{\pi}{2R}t)-\frac{\pi}{2R}\sin(\frac{\pi}{2R}t)]+
         e^{-\frac{n}{2}t}[\frac{n^2}{4}\cos(\frac{\pi}{2R}t)
    \\ &\ \ \ \ +\frac{n\pi}{4R}\sin(\frac{\pi}{2R}t)+\frac{n\pi}{4R}\sin(\frac{\pi}{2R}t)-\frac{\pi^2}{4R^2}\cos(\frac{\pi}{2R}t)]
    \\ &=e^{-\frac{n}{2}t}[-\frac{n^2}{4}\cos(\frac{\pi}{2R}t)-\frac{\pi^2}{4R^2}\cos(\frac{\pi}{2R}t)]
    \\ &=-(\frac{n^2}{4}+\frac{\pi^2}{4R^2})f
  \end{aligned}
\end{equation}
which means that $u$ is an eigenfunction and $\lambda_1(B(N,R))=\frac{n^2}{4}+\frac{\pi^2}{4R^2}.$ Then
the monotonicity of the first Dirichlet eigenvalue and (\ref{5.1}) would imply that
\begin{equation}\label{5.3}
  \lambda_1(B(p,R))=\frac{n^2}{4}+\frac{\pi^2}{4R^2}+O(R^{-3}), \ \ R\rightarrow +\infty.
\end{equation}
\begin{example}
Let $(M,g)=(\mathbb{R}\times N,dt^2+\cosh^2(t) g_N)$ be a $3-$dimensional manifold where $(N,g_N)$ is a compact surface with Gaussian
curvature bounded from below by $-1.$ Then $Ric[g]\geq -2$ and $\lambda_1(M)=1.$
\end{example}
As discussed above, we only need to calculate the first Dirichlet eigenvalue of $B(N,R)=(-R,R)\times N.$ Let $f=\frac{\cos(\frac{\pi}{2R}t)}{\cosh (t)},$ then
$$f>0\ \ in\ B(N,R),\ \ \ \ f=0 \ \ on\ \partial B(N,R).$$
Furthermore,
\begin{equation}
f'(t)=-\frac{\pi}{2R}\frac{\sin(\frac{\pi}{2R}t)}{\cosh(t)}-\tanh(t)\cdot f
\end{equation}
and
\begin{equation}
f''(t)=(-\frac{\pi^2}{4R^4}+\tanh^2(t)-\frac{1}{\cosh^2(t)})f+\frac{\pi}{R}\frac{\sinh(t)}{\cosh^2(t)}\sin(\frac{\pi}{2R}t)
\end{equation}
and hence
\begin{equation}
  \begin{aligned}
    \Delta f&=f'(t)\Delta t+f''(t)|\nabla t|^2=f'(t)\cdot 2\tanh (t)+f''(t)
    \\ &=-\frac{\pi}{R}\frac{\sinh(t)}{\cosh^2(t)}\sin(\frac{\pi}{2R}t)-2\tanh^2(t)\cdot f
    \\ &\ \ \ \ +(-\frac{\pi^2}{4R^4}+\tanh^2(t)-\frac{1}{\cosh^2(t)})f+\frac{\pi}{R}\frac{\sinh(t)}{\cosh^2(t)}\sin(\frac{\pi}{2R}t)
    \\ &=(-\frac{\pi^2}{4R^4}-\tanh^2(t)-\frac{1}{\cosh^2(t)})f
    \\ &=-(1+\frac{\pi^2}{4R^2})f
  \end{aligned}
\end{equation}
We obtain that $\lambda_1(B(N,R))=1+\frac{\pi^2}{4R^2}$ and hence
  for any $p\in M,$
  \begin{equation}\label{5.7}
  \lambda_1(B(p,R))=1+\frac{\pi^2}{4R^2}+O(R^{-3}), \ \ R\rightarrow +\infty.
\end{equation}
Theorem \ref{thm5.1} and (\ref{5.3}), (\ref{5.7}) all lead naturally to Proposition 1.2.

\bibliographystyle{plain}%

\bibliography{bibfile}

\begin{thebibliography}{10}

\bibitem{artamoshin2016lower}
Sergei Artamoshin.
\newblock Lower bounds for the first dirichlet eigenvalue of the laplacian for
  domains in hyperbolic space.
\newblock In {\em Mathematical Proceedings of the Cambridge Philosophical
  Society}, volume 160, pages 191--208. Cambridge University Press, 2016.

\bibitem{cai2000boundaries}
Mingliang Cai and Gregory~J Galloway.
\newblock Boundaries of zero scalar curvature in the ads/cft correspondence.
\newblock {\em arXiv preprint hep-th/0003046}, 2000.

\bibitem{cheng1975eigenvalue}
Shiu-Yuen Cheng.
\newblock Eigenvalue comparison theorems and its geometric applications.
\newblock {\em Mathematische Zeitschrift}, 143:289--297, 1975.

\bibitem{gage1980upper}
Michael~E Gage.
\newblock Upper bounds for the first eigenvalue of the laplace-beltrami
  operator.
\newblock {\em Indiana University Mathematics Journal}, 29(6):897--912, 1980.

\bibitem{guillarmou2010spectral}
Colin Guillarmou and Jie Qing.
\newblock Spectral characterization of poincar{\'e}--einstein manifolds with
  infinity of positive yamabe type.
\newblock {\em International Mathematics Research Notices}, 2010(9):1720--1740,
  2010.

\bibitem{lee1994spectrum}
John~M Lee.
\newblock The spectrum of an asymptotically hyperbolic einstein manifold.
\newblock {\em Communications in Analysis and Geometry}, 3(2):253--271, 1995.

\bibitem{li2001complete}
Peter Li and Jiaping Wang.
\newblock Complete manifolds with positive spectrum.
\newblock {\em Journal of Differential Geometry}, 58(3):501--534, 2001.

\bibitem{li2002complete}
Peter Li and Jiaping Wang.
\newblock Complete manifolds with positive spectrum, ii.
\newblock {\em Journal of Differential Geometry}, 62(1):143--162, 2002.

\bibitem{mazzeo1988hodge}
Rafe Mazzeo.
\newblock The hodge cohomology of a conformally compact metric.
\newblock {\em Journal of differential geometry}, 28(2):309--339, 1988.

\bibitem{mazzeo1991unique}
Rafe Mazzeo.
\newblock Unique continuation at infinity and embedded eigenvalues for
  asymptotically hyperbolic manifolds.
\newblock {\em American Journal of Mathematics}, 113(1):25--45, 1991.

\bibitem{mckean1970upper}
Henry~P McKean.
\newblock An upper bound to the spectrum of $\triangle$ on a manifold of
  negative curvature.
\newblock {\em Journal of Differential Geometry}, 4(3):359--366, 1970.

\bibitem{savo2009lowest}
Alessandro Savo.
\newblock On the lowest eigenvalue of the hodge laplacian on compact,
  negatively curved domains.
\newblock {\em Annals of Global Analysis and Geometry}, 35:39--62, 2009.

\bibitem{wang2001conformally}
Xiaodong Wang.
\newblock On conformally compact einstein manifolds.
\newblock {\em Mathematical Research Letters}, 8(5):671--688, 2001.

\bibitem{wang2002new}
Xiaodong Wang.
\newblock A new proof of lee’s theorem on the spectrum of conformally compact
  einstein manifolds.
\newblock {\em Communications in Analysis and Geometry}, 10(3):647--651, 2002.

\bibitem{witten1999connectedness}
Edward Witten and ST~Yau.
\newblock Connectedness of the boundary in the ads/cft correspondence.
\newblock {\em Advances in Theoretical and Mathematical Physics}, 3(6):1--19,
  1999.

\end{thebibliography}

\noindent{Xiaoshang Jin}\\
  School of mathematics and statistics, Huazhong University of science and technology, Wuhan, P.R. China. 430074
 \\Email address: jinxs@hust.edu.cn

\end{document}